\documentclass{amsart}

\usepackage{latexsym}
\usepackage{amsmath}
\usepackage{amssymb}
\usepackage{amsthm}
\usepackage{amscd}
\usepackage{color}

\newtheorem{theorem}{Theorem}[section]
\newtheorem{lemma}[theorem]{Lemma}

\newtheorem{proposition}[theorem]{Proposition}

\theoremstyle{definition}

\newtheorem{remark}[theorem]{Remark}
\newtheorem{definition}[theorem]{Definition}
\newtheorem{example}[theorem]{Example}

\theoremstyle{remark}

\numberwithin{equation}{section}

\begin{document}

\title{Boundedness of composition operators from Lorentz spaces to Orlicz spaces}

\author{Naoya Hatano, Masahiro Ikeda and Ryota Kawasumi
}

\address[Naoya Hatano]{Department of Department of Mathematics, Chuo University, 1-13-27, Kasuga, Bunkyo-ku, Tokyo 112-8551, Japan,}

\address[Masahiro Ikeda]{Center for Advanced Intelligence Project, RIKEN, Japan/Department of Mathematics, Faculty of Science and Technology, Keio University, 3-14-1 Hiyoshi, Kohoku-ku, Yokohama 223-8522, Japan,}

\address[Ryota Kawasumi]{Faculty of Business Administration, Kobe Gakuin University, 1-1-3 Minatojima, Chuo-ku, Kobe, Hyogo, 650-8586, Japan}

\email[Naoya Hatano]{n.hatano.chuo@gmail.com}

\email[Masahiro Ikeda]{masahiro.ikeda@riken.jp}

\email[Ryota Kawasumi]{kawasumi@ba.kobegakuin.ac.jp}

\begin{abstract}
The boundedness (continuity) of composition operators from some function space to another one is significant, though there are few results about this problem.
Thus, in this study, we provide necessary and sufficient conditions on the boundedness of composition operators from Lorentz spaces to Orlicz spaces. We also give a counter example of a mapping which implies unboundedness of the composition operators from a Lebesgue space $L^p$ to another Lebesgue space $L^q$ with $p>q$.


We emphasize that the measure spaces associated with the Lorentz space may be different from those associated with the Orlicz spaces.
We give more examples and counterexamples of the composed mappings in the conditions satisfying our main results.
\end{abstract}


\maketitle

\section{Introduction}

Throughout this paper, the simbol $L^0(X,\mu)$ is defined by the space of all measurable functions over the measure space $(X,\mu)$. Let $X=(X,\mu)$ and $Y=(Y,\nu)$ be $\sigma$-finite measure spaces.

\begin{definition}
Let $\tau:Y\to X$ be a measurable map, and assume that $\tau$ is nonsingular, namely, $\nu(\tau^{-1}(E))=0$ for each $\mu$-null set $E$.
The composition operator $C_\tau$ is defined by
\[
C_\tau f(y)
\equiv
f(\tau(y)),
\quad
y\in Y,
\]
for all $f\in L^0(X,\mu)$.
\end{definition}

Hereafter $\tau$ stands for a measurable and nonsingular map from $Y$ to $X$.

The boundedness of the composition operators on the Lebesgue spaces is classical and well known by Singh \cite{Singh76}.
After that many authors gave some generalizations for some function spaces \cite{ADV07, Evseev17,CHRL04,Kumar97,EvMe19-1,EvMe19-2,RaSh12}.
Especially, the boundedness on the Orlicz spaces is given by Cui, Hudzik, Kumar and Maligranda in \cite{CHRL04}.
In this paper, we aim the boundedness from one function space to another one.
When $0<p\le q<\infty$, the necessary and sufficient condition of the boundedness of the composition operator $C_\tau$ from $L^p(X,\mu)\to L^q(Y,\nu)$ is for all $\mu$-measurable sets $E\subset X$,
\[
\nu(\tau^{-1}(E))^p
\le C
\mu(E)^q.
\]
Moreover, when $0<q<p<\infty$, there is a counterexample of $\tau$ that the composition operator $C_\tau$ is bounded from $L^p(X,\mu)$ to $L^q(Y,\nu)$. As mentioned in \cite{Evseev17} in the framework of the Lorentz spaces, we can show the boundedness of $C_\tau:L^{p,r}(X,\mu)\to L^{q,r}(Y,\nu)$ without the condition $0<p\le q<\infty$,
where $L^{p,r}(X,\mu)$ and $L^{q,r}(Y,\nu)$ are Lorentz spaces with $r\in (0,\infty]$.

In this paper, we provide necessary and sufficient conditions on the boundedness of composition operators from Lorentz spaces to Orlicz spaces. We emphasize that the measure spaces associated with the Lorentz space may be different from those associated with the Orlicz spaces.
Additionally, we give some examples for the cases which $(X,\mu)$ and $(Y,\nu)$ the counting measure and Lebesgue measure spaces, respectively.
According to these examples, we conject to be able to consider the pullback from some continuous data to some discrete data in the machine learning.



Here and below, we will mention the main theorems.
The mapping $\Phi:[0,\infty)\to[0,\infty)$ is called Young function when it satisfies the following conditions:
\begin{itemize}
\item[(1)] $\Phi$ is a positive function on $(0,\infty)$.
\item[(2)] $\Phi$ is a convex function.
\item[(3)] $\lim\limits_{t\downarrow0}\Phi(t)=\Phi(0)=0$.
\end{itemize}
Hereafter we assume that $\Phi$ is a Young function.

\begin{definition}\label{def:Orlicz-space}
Let $\Phi:[0,\infty)\to[0,\infty)$ satisfy that $\Phi((\cdot)^{1/q})$ is a Young function for some $0<q<\infty$.
The Orlicz space $L^\Phi(X,\mu)$ is defined by
\[
L^\Phi(X,\mu)
\equiv
\left\{
f\in L^0(X,\mu)
\,:\,
\int_X\Phi(\epsilon |f(x)|)\,{\rm d}\mu(x)
<\infty, \quad \text{for every}\quad \epsilon >0
\right\}
\]
endowed with the {\rm (quasi-)}norm
\[
\|f\|_{L^\Phi(X)}
\equiv
\inf\left\{
\lambda>0\,:\,
\int_X
\Phi\left(\frac{|f(x)|}\lambda\right)
\,{\rm d}\mu(x)
\le1
\right\}.
\]
\end{definition}
Orlicz spaces are (quasi-)Banach spaces.
If $\Phi(t) = t^p$, $0<p<\infty$, then $L^\Phi(X, \mu)$ is $L^p(X,\mu)$.
Moreover, when $\Phi(t)=t\log(3+t)$ or $e^t-1$, the function space $L^\Phi(X,\mu)$ is an Orlicz space which is given as $L\log L(X,\mu)$ or $\exp L(X,\mu)$.

\begin{remark}
$\sigma$-finiteness is not necessary in the above definition.    
\end{remark}

\begin{definition}\label{def:Lorentz-space}
Let $0<p\le\infty$ and $0<q\le\infty$.
The Lorentz space $L^{p,q}(X,\mu)$ is defined by the space of all $f\in L^0(X,\mu)$ with the finite quasi-norm
\[
\|f\|_{L^{p,q}(X)}
\equiv
\begin{cases}
\displaystyle
\left(
\int_0^\infty
\left[t\mu(\{x\in X\,:\,|f(x)|>t\})^{\frac1p}\right]^q
\,\frac{{\rm d}t}t
\right)^{\frac1q},
& q<\infty, \\
\displaystyle
\sup_{t>0}
t\mu(\{x\in X\,:\,|f(x)|>t\})^{\frac1p},
&q=\infty.
\end{cases}
\]
\end{definition}

Since $L^{p,p}(X,\mu)=L^p(X,\mu)$, the Lorentz spaces are generalization for the Lebesgue spaces.
Additionally, the Lorentz spaces $L^{p,\infty}(X,\mu)$ stand for the weak Lebesgue spaces.
When $p=\infty$, the Lorentz space $L^{\infty,q}(X,\mu)$ is equivalent to $L^\infty(X,\mu)$, formally.
The Lorentz spaces $L^{p,q}(X,\mu)$ are quasi-Banach spaces.
But it is well known that $p,q>1$ implies the Lorentz spaces $L^{p,q}(X,\mu)$ are normable (cf. \cite{Grafakos14-1}).


The following is our main result.

\begin{theorem}\label{main:Lorentz-Orlicz}
Let $0<p<\infty$ and $\Phi:[0,\infty)\to[0,\infty)$ be a Young function.
Then, the composition operator $C_\tau$ is bounded from $L^{p,1}(X,\mu)$ to $L^\Phi(Y,\nu)$ if and only if there exists a constant $D\ge1$ such that for all $\mu$-measurable sets $E\subset X$,
\[
\nu(\tau^{-1}(E))
\le
\left\{
\Phi\left(\frac1{D\mu(E)^{\frac1p}}\right)
\right\}^{-1}.
\]
\end{theorem}

By this theorem and the properties
\[
\|f\|_{L^{p,q}(X)}
=
\||f|^q\|_{L^{p/q,1}(X)}^{\frac1q},
\quad
\|f\|_{L^\Phi(X)}
=
\||f|^q\|_{L^{\Phi((\cdot)^{1/q})}(X)}^{\frac1q}
\]
for each $f\in L^0(X,\mu)$, the following theorem can be obtained immediately.
Then, we can see the following generalization for the previous main result.

\begin{theorem}\label{main:Lorentz-Orlicz-2}
Let $0<p,q<\infty$ and $\Phi:[0,\infty)\to[0,\infty)$.
If $\Phi((\cdot)^{1/q})$ is a Young function, then, the composition operator $C_\tau$ is bounded from $L^{p,q}(X,\mu)$ to $L^\Phi(Y,\nu)$ if and only if there exists a constant $D\ge1$ such that for all $\mu$-measurable sets $E\subset X$,
\begin{equation}\label{eq:vol}
\nu(\tau^{-1}(E))
\le
\left\{
\Phi\left(\frac1{D\mu(E)^{\frac1p}}\right)
\right\}^{-1}.
\end{equation}
\end{theorem}

We remark that the boundedness of the composition operators from some Orlicz spaces to the different Orlicz spaces is interesting. Chawziuk et al.\ \cite{CCEHK16} (Musielak-Orlicz space) and Labuschagne and Majewski \cite{LaMa11} gave the necessary and sufficient condition on $\tau$ of the boundedness of the composition operator from an Orlicz space to another Orlicz one under the assumption that $\tau$ is surjective.

We will use standard notation for inequalities.
We use $C$ to denote a positive constant that may vary from one occurrence to another.
If $f\le Cg$, then we write $f\lesssim g$ or $g\gtrsim f$, and if $f\lesssim g\lesssim f$, then we write $f\sim g$.

We organize the remaining part of the paper as follows:
We write the preliminaries for the proof of theorems in Section \ref{s:Pre}.
We prove the main theorem in Section \ref{s:proof-main}.
In Section \ref{exmple}, we give examples of $\tau$ in the condition of the main theorems.

\section{Preliminaries}\label{s:Pre}

In this section, we introduce some Lemmas for Orlicz and Lorentz spaces to prove the main theorem.
Especially, as a key lemma, we use the following layer cake formula with respect to Orlicz space, and then we can use the volume estimate in the proof of Theorem \ref{main:Lorentz-Orlicz}, directly.

\begin{lemma}\label{lem:Young}
For a Young function $\Phi$, the function $\varphi$ satisfying
\[
\Phi(t)
=
\int_0^t\varphi(s)\,{\rm d}s
\]
is said left derivative of $\Phi$.
Then, the following assertions hold{\rm :}
\begin{itemize}
\item[{\rm (1)}] $\displaystyle
\int_X\Phi(|f(x)|)\,{\rm d}\mu(x)
=
\int_0^\infty\varphi(t)\mu(\{x\in X\,:\,|f(x)|>t\})\,{\rm d}t
$. 

\item[{\rm (2)}] $\displaystyle
\frac{\Phi(t)}t\le\varphi(t)\le\frac{\Phi(2t)}t$.
\end{itemize}
\end{lemma}
For (1), see ~\cite[Proposition~1.1.4]{Grafakos14-1} and for (2), see~\cite[P.20]{Rao-Ren1991}.

Additionally, to prove the only if part of Theorem \ref{main:Lorentz-Orlicz}, we use the Orlicz norm and Lorentz quasi-norm for the indicator functions, as follows.

\begin{lemma}
{\cite[Corllary~7 in P.78]{Rao-Ren1991}}\label{lem:Orlicz-indicator}
Let $\Phi$ be a Young function.
Then, for all measurable sets $E\subset X$, the following holds:
\[
\|\chi_E\|_{L^\Phi(X)}
=
\left\{
\Phi^{-1}\left(\frac1{\mu(E)}\right)
\right\}^{-1}.
\]
\end{lemma}

\begin{lemma}[{\cite[Example 1.4.8]{Grafakos14-1}}]\label{lem:Lorentz-indicator}
Let $0<p<\infty$ and $0<q\le\infty$.
Then, for all $\mu$-measurable sets $E\subset X$,
\[
\|\chi_E\|_{L^{p,q}(X)}
=
q^{-\frac1q}\mu(E)^{\frac1p},
\]
where if $q=\infty$, one assumes that $q^{-1/q}=1$.
\end{lemma}

To normalize, we use the following embeddings for the Lorentz spaces (see \eqref{eq:240902}).

\begin{lemma}[{\cite[Proposition 1.4.10]{Grafakos14-1}}]\label{lem:Lorentz-embedding}
Let $0<p<\infty$ and $0<q_1\le q_2\le\infty$.
Then,
\[
L^{p,q_1}(X,\mu)
\hookrightarrow
L^{p,q_2}(X,\mu).
\]
\end{lemma}

\section{Proof of the main theorem}\label{s:proof-main}

In this section, we will prove Theorem \ref{main:Lorentz-Orlicz}.
In this proof, for a measurable function $f$, we write the distribution functions as
\[
\mu_f(t)
\equiv
\mu(\{x\in X\,:\,|f(x)|>t\}),
\quad
\nu_f(t)
\equiv
\nu(\{x\in Y\,:\,|f(x)|>t\}),\ \ \ t\ge 0.
\]

At first, we prove the \lq\lq if" part.
Since the embedding $L^{p,1}(X,\mu)\hookrightarrow L^{p,\infty}(X,\mu)$ holds (see Lemma \ref{lem:Lorentz-embedding}), we only consider the case $\|f\|_{L^{p,\infty}(X)}\le(2D)^{-1}$ by normalization. Using Lemma \ref{lem:Young} (1) and the assumption, we have
\begin{align*}
\int_Y\Phi(|C_\tau f(y)|)\,{\rm d}\nu(y)
=
\int_0^\infty\varphi(t)\nu_{C_\tau f}(t)\,{\rm d}t
\le
\int_0^\infty
\varphi(t)\left\{\Phi\left(\frac1{D\mu_f(t)^{\frac1p}}\right)\right\}^{-1}
\,{\rm d}t.
\end{align*}
Since the mapping $t\mapsto\Phi(t)/t$ is increasing, we deduced that
\begin{align}\label{eq:240902}
\Phi\left(\frac1{D\mu_f(t)^{\frac1p}}\right)
=
\left\{
2D\cdot t\mu_f(t)^{\frac1p}
\Phi\left(\frac{2t}{2D\cdot t\mu_f(t)^{\frac1p}}\right)
\right\}
\cdot
\frac1{2D\cdot t\mu_f(t)^{\frac1p}}
\ge
\frac{\Phi(2t)}{2D\cdot t\mu_f(t)^{\frac1p}}.
\end{align}
Thus, by Lemma \ref{lem:Young} (2),
\begin{align*}
\int_Y\Phi(|C_\tau f(y)|)\,{\rm d}\nu(y)
&\le
\int_0^\infty
\varphi(t)\left\{\frac{\Phi(2t)}{2D\cdot t\mu_f(t)^{\frac1p}}\right\}^{-1}
\,{\rm d}t
\le
2D
\int_0^\infty\mu_f(t)^{\frac1p}\,{\rm d}t\\
&=
2D
\|f\|_{L^{p,1}(X)}.
\end{align*}

Next, we prove the \lq\lq only if" part.
Fix a $\mu$-measurable set $E$.
Then, assuming the boundedness of the composition operator $C_\tau$ from $L^{p,1}(X,\mu)$ to $L^\Phi(Y,\nu)$ with the operator norm $D$, we have
\[
\|C_\tau\chi_E\|_{L^\Phi(Y)}
\le D
\|\chi_E\|_{L^{p,1}(X)}.
\]
By $C_\tau\chi_E=\chi_{\tau^{-1}(E)}$, using Lemmsa \ref{lem:Orlicz-indicator} and \ref{lem:Lorentz-indicator}, we have
\[
\Phi^{-1}\left(\frac1{\nu(\tau^{-1}(E))}\right)
\ge
\frac1{D\mu(E)^{\frac1p}},
\]
or equivalently,
\[
\nu(\tau^{-1}(E))
\le
\left\{\Phi\left(\frac1{D\mu(E)^{\frac1p}}\right)\right\}^{-1},
\]
as desired.


\section{Examples}
\label{exmple}

In this section, we give some examples and counterexamples for a map $\tau$ included in our main results.

\begin{example}
If we set $\tau(y)=y$ for $y\in{\mathbb R}$, the volume estimate
\[
|\tau^{-1}(E)|
=
|E|
\]
holds for any Lebesgue measurable set $E\subset{\mathbb R}$.
If $p\ge1$, then the composition operator $C_\tau$ is bounded from $L^{p,1}({\mathbb R})$ to $L^p({\mathbb R})$.
But when $p<1$, the boundedness $C_\tau:L^{p,1}({\mathbb R})\to L^p({\mathbb R})$ fails.
In fact, choosing
\[
f(x)
=
(1+|x|)^{-\frac1p}(\log(3+|x|))^{-\frac1p},
\quad
x\in{\mathbb R},
\]
we have $f=C_\tau f\in L^{p,1}({\mathbb R})\setminus L^p({\mathbb R})$.
Here, see \cite[Excercise 1.4.8]{Grafakos14-1} for details of the example $f(x)$.
\end{example}

\begin{example}
Let $0<p,q<\infty$, $0<r\le\infty$, $(X,\mu)=({\mathbb Z},\sharp)$ and $(Y,\nu)=({\mathbb R},|\cdot|)$, and set $\tau(y)=[|y|^{p/q}]$ for $y\in{\mathbb R}$, where $[\cdot]$ is the Gauss symbol.
Note that for each $k\in\{0\}\cup{\mathbb N}$,
\[
\tau^{-1}(\{k\})
=
\left(-(k+1)^{\frac qp},-k^{\frac qp}\right]
\cup
\left[k^{\frac qp},(k+1)^{\frac qp}\right).
\]
\begin{itemize}
\item[{\rm (1)}] Let
$
E=\{n,n+1,\ldots,m-1\}\subset\{0\}\cup{\mathbb N}
$ ($n<m$),
\[
\sharp E=m-n,
\quad
|\tau^{-1}(E)|^{\frac pq}
\le
\left\{
2\sum_{j=n}^{m-1}
\left[(j+1)^{\frac qp}-j^{\frac qp}\right]
\right\}^{\frac pq}
=
\left\{2
\left(m^{\frac qp}-n^{\frac qp}\right)
\right\}^{\frac pq}
\]
Then, $p\ge q$ implies
\[
\left(m^{\frac qp}-n^{\frac qp}\right)^{\frac pq}
\le
m-n,
\]
and then
\[
|\tau^{-1}(E)|^{\frac pq}\lesssim\sharp E.
\]
But, when $p<q$, this estimate fails.
In fact, when $m=n+1$, there exists $\theta\in(0,1)$ such that
\begin{align*}
\frac12
\left(\frac{|\tau^{-1}(E)|^{\frac pq}}{\sharp E}\right)^{\frac qp}
=
(n+1)^{\frac qp}-n^{\frac qp}
=
\frac qp(n+\theta)^{\frac qp-1},
\end{align*}
and then,
\[
\lim_{n\to\infty}
\frac{|\tau^{-1}(E)|^{\frac pq}}{\sharp E}
=
\infty.
\]

\item[{\rm (2)}] Assume that $p\ge q$, and let $E$ be a general finite subset of ${\mathbb Z}$.
Then, by the convexity of $t\mapsto t^{p/q}$,
\[
|\tau^{-1}(E)|
\le
2|\tau^{-1}(\{0,1,2,\ldots,k-1\})|
\]
holds, where $k=\sharp E$.
It follows from (1) as $n=0$ and $m=k$ that
\[
|\tau^{-1}(E)|^{\frac pq}\lesssim\sharp E
\]
can be obtained, and therefore, the composition operator $C_\tau$ is bounded from $L^{p,r}({\mathbb Z})$ to $L^{q,r}({\mathbb R})$.
Here, $L^{p,r}({\mathbb Z})$ stands for the Lorentz sequence $\ell^{p,r}({\mathbb Z})$.

\item[{\rm (3)}] When $p>q$, the boundedness from the Lebesgue sequence space $L^p({\mathbb Z})=\ell^p({\mathbb Z})$ to $L^q({\mathbb R})$ fails.
In fact, choosing
\[
f(n)
=
(1+|n|)^{-\frac1p}(\log(3+|n|))^{-\frac1q},
\quad
n\in{\mathbb Z},
\]
we have
\begin{align*}
|C_\tau f(y)|
&=
\left(1+\left[|y|^{\frac pq}\right]\right)^{-\frac1p}
\left(\log\left(3+\left[|y|^{\frac pq}\right]\right)\right)^{-\frac1q}\\
&\ge
(1+|y|)^{-\frac1q}(\log(3+|y|))^{-\frac1q}
\times\left(\frac pq\right)^{-\frac1q}.
\end{align*}
Then, $f \in\ell^p({\mathbb Z})$ but $C_\tau f\notin L^q({\mathbb R})$.
Here, see \cite[Excercise 1.4.8]{Grafakos14-1} for details of the example $f(n)$.

\item[{\rm (4)}] Let $K=\{1,2,\ldots,k\}$.
Then, the estimate
\[
|\tau^{-1}(E)|^{\frac pq}
\lesssim
\sharp E
\]
always holds for any $E\subset K$.
Especially, since $\ell^p(K)\cong{\mathbb C}^k$ independent of $p$, the boundedness of the composition operator $C_\tau:{\mathbb C}^k\to L^q({\mathbb R})$ can be obtained.
\end{itemize}
\end{example}

\begin{example}
Let $p\ge1$, $\Phi(t)=t\log(3+t)$, and let $\tau:{\mathbb R}\to{\mathbb Z}$ be
\[
\tau(y)
=
\left[
\left\{
\Phi^{-1}\left(\frac1{|y|}\right)
\right\}^{-p}
\right],
\quad y\in{\mathbb R}.
\]
Since the volume estimate
\[
|\tau^{-1}(E)|
\lesssim
\left\{
\Phi\left(\frac1{(\sharp E)^{\frac1p}}\right)
\right\}^{-1}
\]
for each $E\subset{\mathbb Z}$, therefore Theorem \ref{main:Lorentz-Orlicz} the composition operator $C_\tau$ is bounded from $\ell^{p,1}({\mathbb Z})$ to $L\log L({\mathbb R})$.
Note that for each $k\in\{0\}\cup{\mathbb N}$,
\[
\tau^{-1}(\{k\})
=
\left(
-\left\{\Phi\left(\frac1{(k+1)^{\frac1p}}\right)\right\}^{-1},
-\left\{\Phi\left(\frac1{k^{\frac1p}}\right)\right\}^{-1}
\right]
\cup
\left[
\left\{\Phi\left(\frac1{k^{\frac1p}}\right)\right\}^{-1},
\left\{\Phi\left(\frac1{(k+1)^{\frac1p}}\right)\right\}^{-1}
\right).
\]
\begin{itemize}
\item[{\rm (1)}] Let
$
E=\{n,n+1,\ldots,m-1\}\subset{\mathbb N}
$ ($n<m$).
Note that
\[
\sharp E=m-n,
\]
\begin{align*}
|\tau^{-1}(E)|
&=
2\sum_{j=n}^m
\left[
\left\{
\Phi\left(\frac1{(k+1)^{\frac1p}}\right)
\right\}^{-1}
-
\left\{
\Phi\left(\frac1{k^{\frac1p}}\right)
\right\}^{-1}
\right]\\
&=
2\left[
\left\{
\Phi\left(\frac1{m^{\frac1p}}\right)
\right\}^{-1}
-
\left\{
\Phi\left(\frac1{n^{\frac1p}}\right)
\right\}^{-1}
\right].
\end{align*}

\item[{\rm (2)}] Let $E$ be a general finite subset of ${\mathbb Z}$.
Then, assuming that the mapping
$
t\mapsto\{\Phi^{-1}(1/t)\}^{-p}
$
is equivalent to some convex function at $t>1$, we have
\[
|\tau^{-1}(E)|
\lesssim
|\tau^{-1}(\{n,n+1,\ldots,m-1\})|
\]
holds, where the natural numbers $1<n<m$ are taken by $n=\min E$ and $\sharp E=m-n$.
It follows from (1) that
\begin{align*}
|\tau^{-1}(E)|
\lesssim
\left\{
\Phi\left(\frac1{(k+1)^{\frac1p}}\right)
\right\}^{-1}
\lesssim
\left\{
\Phi\left(\frac1{(\sharp E)^{\frac1p}}\right)
\right\}^{-1}
\end{align*}
can be obtained, and therefore, the composition operator $C_\tau$ is bounded from $\ell^{p,1}({\mathbb Z})$ to $L\log L({\mathbb R})$.

\item[{\rm (3)}] Let us confirm that the mapping $t\mapsto\{\Phi^{-1}(1/t)\}^{-p}$ is equivalent to the convex function $t\mapsto t^p$ on $(1,\infty)$, that is, for all $t>1$,
\[
\left\{
\Phi^{-1}\left(\frac1t\right)
\right\}^{-p}
\sim
t^p
\]
holds.
It is well known by O'Neil \cite{ONeil65} that for all $t>0$,
\[
t\le\Psi^{-1}(t)\widetilde{\Psi}^{-1}(t)\le2t,
\]
where $\widetilde{\Psi}$ is a complementary function of the Young function $\Psi$, that is,
\[
\widetilde{\Psi}(t)
\equiv
\sup\{st-\Psi(s)\,:\,s\ge0\},
\quad
t\ge0,
\]
and $\Psi^{-1}$ is a generalized inverse function of the nondecreasing function, that is,
\[
\Psi^{-1}(t)
\equiv
\inf\{s\ge0\,:\,\Psi(s)>t\},
\quad
t\ge0.
\]
Additionally, the complementary function $\widetilde{\Phi}$ is provided by the first, third authors and Ono \cite[Appendix B]{HKO23} as follows{\rm :}
\[
\widetilde{\Phi}(t)
\sim
\begin{cases}
0, & 0\le t\le\log3, \\
(t-\log3)^2, & \log3<t\le3, \\
e^t, & t>3.
\end{cases}
\]
Then,
\[
\widetilde{\Phi}^{-1}(t)
\sim
\begin{cases}
\sqrt{t}+\log 3, & 0<t<(3-\log3)^2, \\
\log t, & t\ge(3-\log3)^2.
\end{cases}
\]
It follows that for each $t>1$,
\begin{align*}
\left\{
\Phi^{-1}\left(\frac1t\right)
\right\}^{-p}
\sim
\left\{
t\widetilde{\Phi}^{-1}\left(\frac1t\right)
\right\}^p
\sim
(\sqrt{t}+t\log3)^p
\sim
t^p,
\end{align*}
and then $t>1$ implies that the mapping $t\mapsto\{\Phi^{-1}(1/t)\}^{-p}$ is equivalent to the convex function $t\mapsto t^p$.
\end{itemize}
\end{example}

\begin{example}
Let $p\ge 1$, and let $\tau:{\mathbb R}\to{\mathbb Z}$ be
\[
\tau(y)
=
\left[
\left\{
\log\left(1+\frac1{|y|}\right)
\right\}^{-p}
\right],
\quad y\in{\mathbb R}.
\]
Then, the volume estimate
\begin{equation}\label{eq:vol-exp}
|\tau^{-1}(E)|
\lesssim
\left\{
\exp\left(
\frac1{(\sharp E)^{\frac1p}}
\right)
-1
\right\}^{-1}
\end{equation}
holds for each $E\subset{\mathbb Z}$, and therefore Theorem \ref{main:Lorentz-Orlicz} implies that the composition operator $C_\tau$ is bounded from $\ell^{p,1}({\mathbb Z})$ to $\exp L({\mathbb R})$.
In fact, since the mapping $t\mapsto\{\log(1+1/t)\}^{-p}$ is equivalent to the convex function $t\mapsto t^p$ on $(1,\infty)$, it follows from the similar reason of the previous example that the equation \eqref{eq:vol-exp} holds.
That equivalency follows from the following argument: when $t>1$, there exists $\theta=\theta(t)\in(0,1)$ such that
\begin{align*}
\log\left(1+\frac1t\right)
=
\log(t+1)-\log t
=
\frac1{t+\theta}
\sim
\frac1t,
\end{align*}
and then,
\[
\left\{
\log\left(1+\frac1t\right)
\right\}^{-p}
\sim
t^p.
\]
\end{example}

\section{Remarks}

\subsection{Remarks on the settings of Theorem \ref{main:Lorentz-Orlicz}}

In this subsection, we give some remarks on the $\tau:Y\rightarrow X$ satisfying the volume estimate \eqref{eq:vol} with $(X,\mu)=(Y,\nu)=({\mathbb R},|\cdot|)$. Especially we consider $n=1$ and $\tau:\mathbb{R}\rightarrow\mathbb{R}$.


At first we consider the case $\Phi(t)=t^q$ ($q>p$) as follows.

\begin{proposition}\label{prop:240823}
Let $0<p<q<\infty$.
Then, there does not exist any continuous map $\tau$ such that the volume estimate
\[
|\tau^{-1}(E)|^{\frac1q}
\lesssim
|E|^{\frac1p},
\quad
E\subset{\mathbb R}.
\]
\end{proposition}
This proposition is a special case of the following proposition.

Here we introduce $\nabla_2$-condition on a Young function $\Phi:[0,\infty)\to[0,\infty)$ as follows.
When there exists a constant $k>1$ such that
\[
\Phi(t)
\le
\frac1{2k}\Phi(kt),
\quad
t>0,
\]
we say that a Young function $\Phi:[0,\infty)\to[0,\infty)$ satisfies $\nabla_2$-condition, denoted by $\Phi\in\nabla_2$. 

\begin{proposition}\label{prop:vol}
Let $0<p<\infty$ and $D>0$, and let $\Phi:[0,\infty)\to[0,\infty)$ satisfying $\Phi((\cdot)^{1/p})\in\nabla_2$.
Then, there does not exist any continuous map $\tau$ such that the volume estimate
\[
|\tau^{-1}(E)|
\le
\left\{\Phi\left(\frac1{D|E|^{\frac1p}}\right)\right\}^{-1},
\quad
E\subset{\mathbb R}
\]
holds.
\end{proposition}

To prove this proposition, we use the following lemmas. 
\begin{lemma}\label{lem:Orlicz-nabla}
If $\Phi\in\nabla_2$, then there exists $\gamma>1$ such that
\[
\frac{\Phi(t)}{t^\gamma}
\lesssim
\frac{\Phi(s)}{s^\gamma},
\quad
0<t<s<\infty.
\]
\end{lemma}
Since this lemma is equivalent to \cite[Lemma~1.2.2]{Kokilashvili-Krbec1991}, we omit the detail.

\begin{lemma}\label{lem:Holder-Orlicz}
Let $f:\mathbb{R}\rightarrow\mathbb{R}$ be a function and $I\subset{\mathbb R}$ be an interval, and let $\Phi\in\nabla_2$.
If there exists $D>0$ such that
\[
|f(x)-f(y)|
\le
\left\{\Phi\left(\frac1{D|x-y|}\right)\right\}^{-1},
\quad
x,y\in I,
\]
then $f$ is a constant function on $I$.
\end{lemma}

\begin{proof}
Fix $x\in I$.
Note that
\[
\left|\frac{f(x+h)-f(x)}h\right|
\le
\frac1{|h|}
\left\{\Phi\left(\frac1{D|h|}\right)\right\}^{-1}
\]
for all $h\in{\mathbb R}\setminus\{0\}$.
By Lemma \ref{lem:Orlicz-nabla}, there exists $\gamma>1$ such that
\[
\frac1{|h|}
\left\{\Phi\left(\frac1{D|h|}\right)\right\}^{-1}
\lesssim
|h|^{\gamma-1}
\]
when $|h|<1$.
It follows that
\[
\left|\frac{f(x+h)-f(x)}h\right|
\to0
\]
as $h\to0$, and then $f(x)$ is differentiable, and $f'(x)=0$.
It follows that $f(x)$ is a constant function on $I$.
\end{proof}

We go back the proof of Proposition \ref{prop:vol}.

\begin{proof}[Proof of Proposition {\rm \ref{prop:vol}}]
Assume that there exists a continuous map $\tau:{\mathbb R}\to{\mathbb R}$ such that the volume estimate
\[
|\tau^{-1}(E)|
\le
\left\{\Phi\left(\frac1{D|E|^{\frac1p}}\right)\right\}^{-1},
\quad
E\subset{\mathbb R}.
\]
It is easy to see that $\tau$ is not constant map, and then we can find an interval $I\ne\emptyset$ such that $\tau:I\to\tau(I)$ is homeomorphism.

Take $E=[x,y)\subset I$ in the volume estimate.
Since the homeomorphism $\tau$ deduces that
\[
\tau^{-1}(E)
=
\tau^{-1}([x,y))
=
[\tau^{-1}(x),\tau^{-1}(y))
\quad or \quad
[\tau^{-1}(y),\tau^{-1}(x)),
\]
the volume estimate gives
\[
|\tau^{-1}(x)-\tau^{-1}(y)|
\le
\left\{
\Phi\left(
\frac1{D|x-y|^{\frac1p}}
\right)
\right\}^{-1},
\]
and then by Lemma \ref{lem:Holder-Orlicz}, $\tau^{-1}$ is a constant function on $I$.
This is contradiction to existence of $\tau^{-1}$.
\end{proof}

\subsection{Remarks on the case $q=\infty$ in Theorem \ref{main:Lorentz-Orlicz-2}}

In this subsection, we assume that $(X,\mu)=({\mathbb R}^n,|\cdot|)$ and $(Y,\nu)=({\mathbb R}^m,|\cdot|)$, and consider some continuous maps $\tau:{\mathbb R}^m\to{\mathbb R}^n$.
Then we see that the boundedness of $C_\tau:L^{p,\infty}({\mathbb R}^n)\to L^\infty({\mathbb R}^m)$ does not hold as follows.

\begin{proposition}
Let $0<p<\infty$.
Then, there does not exist any continuous map $\tau:{\mathbb R}^m\to{\mathbb R}^n$ such that the composition operator $C_\tau:L^{p,\infty}({\mathbb R}^n)\to L^\infty({\mathbb R}^m)$ is bounded.
\end{proposition}

This proposition follows from the following proposition. This is because $L^{p}(\mathbb{R}^n)$ is embedded in $L^{p,\infty}(\mathbb{R}^n)$.

\begin{proposition}
Let $0<p<\infty$.
Then, there does not exist any continuous map $\tau:{\mathbb R}^m\to{\mathbb R}^n$ such that the composition operator $C_\tau:L^p({\mathbb R}^n)\to L^\infty({\mathbb R}^m)$ is bounded.
\end{proposition}

\begin{proof}
Assume that there exists a continuous map $\tau:{\mathbb R}^m\to{\mathbb R}^n$ such that the composition operator $C_\tau:L^p({\mathbb R}^n)\to L^\infty({\mathbb R}^m)$ is bounded.
Then since $f\equiv|\cdot|^{-\gamma}\chi_{B(0,1)}\in L^p({\mathbb R}^n)$ for $\gamma<n/p$, $C_\tau f(x)=|\tau(x)|^{-\gamma}$ is a bounded function.
Here, let $a\in\tau(B(0,1))$.
It follows that by the continuity of $\tau$,
\[
\||\tau(\cdot)-a|^{-\gamma}\chi_{B(0,1)}\|_{L^\infty}
=
\infty.
\]
In fact, for all $M>0$, by the continuity of $\tau$, there exists $\varepsilon>0$ such that for all $x\in B(b,\varepsilon)\cap B(0,1)$,
\[
|\tau(x)-a|^{-\gamma}>M,
\]
where $\tau(b)=a$ for $b\in B(0,1)$.
Hence, this is contradiction.
\end{proof}

{\bf Acknowledgements.} 
The authors are thankful to Professor Yoshihiro Sawano for his giving very helpful comments and advice.
The first-named author (N.H.) was supported by Grant-in-Aid for Research Activity Start-up Grant Number 23K19013.
The second author (M.I.) is supported by the Grant-in-Aid for Scientific Research (C) (No.23K03174 ), Japan Society for the Promotion of Science and JST CREST, No. JPMJCR1913.

\end{document}